\documentclass{amsart}

\usepackage{macros}
\standardmargins\standardpackages\standardrenews
\standardlabeling\standardmakes
\colorcommentstrue

\let\comjohannes\undefined
\allowcomments{\comjohannes}{JS}{Johannes}{magenta}

\draftfalse

\authormumtaz
\authorjohannes
\authordavid

\newcommand{\DDD}{\DD_n(\Psi)}
\newcommand{\DDh}{\DD_n(\Psi)}
\newcommand{\DDp}{\DD_2(\Psi)}
\begin{document}
\title{Diophantine approximation on  curves}

\begin{abstract}Let $g$ be a dimension function. The Generalised Baker-Schmidt Problem (1970) concerns the $g$-dimensional Hausdorff measure ($\HH^g$-measure) of the set of 
	$\Psi$-approximable points on nondegenerate manifolds. The problem relates the `size' of the set of $\Psi$-approximable points with the convergence or divergence of a certain series. In the dual 
	approximation setting, the divergence case has been established by Beresnevich-Dickinson-Velani (2006) for any nondegenerate manifold. The convergence case, however, represents a major challenging open problem and progress thus far has been effectuated 	in limited cases only. In this paper, we discuss and prove several results on $\HH^g$-measure on Veronese curves in any dimension $n$. As a consequence of one of our results, we generalize recent results of Pezzoni [Acta Arith. 193 (2020), no. 3, 269--281] regarding $n=3$. This improvement evolves from a deeper investigation  on  general irreducibility considerations applicable in arbitrary dimension. 
We further investigate the $\HH^g$-measure for convergence on planar curves. We show that the monotonicity assumption on a multivariable approximating function cannot be removed for planar curves.

\end{abstract}
\maketitle


\section{Introduction}
\label{intro}
Let $\Psi:\N\to\Rplus$ be a function such that $\Psi(q)\to 0$ as $q\to \infty$. Khintchine's theorem (1924) is a fundamental result which states that, if $\Psi$ is decreasing, the Lebesgue measure of the set of $\Psi$-approximable numbers
$$W(\Psi):=\{x\in \R: |x-p/q|<\Psi(q) \ {\rm for \ infinitely \ many} \  (p, q)\in \Z\times\N\}$$
is either zero or full if the sum $\sum_{q=1}^\infty q\Psi(q)$  converges or diverges, respectively, see \cite{BDV} for a statement of the modern version of this theorem. The monotonicity of $\Psi$ is only needed 
for the divergence part.  In fact, Duffin--Schaeffer (1941) constructed a counterexample to show that without the monotonicity assumption Khintchine's theorem is false. However, they conjectured a refined statement that holds
for non-monotonic $\Psi$ as well, which has been recently affirmatively proved by Koukoulopoulos--Maynard \cite{KM2019}. 



The higher dimensional theory of Diophantine approximation splits into two different types: the simultaneous and the dual.
Simultaneous Diophantine approximation comprises the component-wise approximation of points $\yy=(y_1,\dots, y_n)^T\in \R^n$ by $n$-tuples of rational points $\{\pp/q:(\pp, q)\in \Z^n\times \N\}$, whereas dual Diophantine approximation consists of the approximation of points $\xx=(x_1,\dots, x_n)\in \R^n$ by `rational' hyperplanes of the form $q_1x_1+\cdots+q_nx_n=p$, where $(p, \qq)= (p, q_1, \ldots q_n)\in\Z\times \Z^n\setminus\{\0\}$. 
The metrical theory for these problems date back to 
Khintchine,	Jarn\'ik, Groshev, and remains to this day a prestigious, well-studied topic in Diophantine approximation.
In this paper we are concerned with a problem in dual approximation. 

\medskip

\noindent{\bf Notation.} 
Throughout, unless otherwise specified, let $\Psi:\N\to[0, \infty)$ be a decreasing function such that
$$\Psi(\|\qq\|) \rightarrow~0 \text{ as } \|\qq\|:=\max(|q_1|, \ldots, |q_n|) \rightarrow~\infty,$$  
referred to as an \emph{approximating function}.  By a dimension function $g$ we mean an increasing continuous function $g:\mathbb{R}^+\to \mathbb{R}^+$ with $g(0)=0$. By $\HH^g$-measure, we mean the $g$-dimensional Hausdorff measure and when $g(r)=r^s$ we
	denote  $\HH^g$ by $\HH^s$.
If $s$ is an integer
then $\HH^s$ is proportional to the standard $s$-dimensional Lebesgue measure.
We refer to Section \ref{hm} below for a brief introduction to Hausdorff measure. For real quantities $A,B$ and a parameter $t$, we write $A \lessless_t B$ if $A \leq c(t) B$ for a constant $c(t) > 0$ that depends on $t$ only (while $A$ and $B$ may depend on other parameters). We write  $A\asymp_{t} B$ if $A\lessless_{t} B\lessless_{t} A$.
If the constant $c>0$ is absolute, we simply write $A\lessless B$ and $A\asymp B$.

\subsection{Dual Diophantine approximation.}
 We consider the dual approximation 
problem with respect to approximating function $\Psi$. Concretely, we are concerned with the set

\begin{equation*}
  \DDD:=\left\{\xx=(x_1,\dots,x_n)\in\R^n:\begin{array}{l}
  |q_1x_1+\cdots+q_nx_n+p|<\Psi(\|\qq\|) \\[1ex]
  \text{for i.m.} \ (p, q_1, \ldots, q_n)\in\Z^{n+1}
                           \end{array}
\right\},
\end{equation*}
where  `i.m.' stands for `infinitely many'.  A vector $\xx \in \R^n$ will be called \emph{$\Psi$-approximable} if it lies in the set $\DDD$.  
The following statement is the most modern result which relates the `size' of the set $\DDD$ in terms of  $\HH^g$-measure to the convergence or divergence of a certain series. This result encompasses contributions from many authors but most importantly works of Jarn\'ik \cite{Jarnik2}, Schmidt \cite{Schmidt8}, Dickinson--Velani \cite{DickinsonVelani}, and Beresnevich--Velani \cite{BeresnevichVelani4}.

%
%

\begin{theorem}\label{JSDV}
Let $\Psi$ be an approximating function. Let $g$ be a dimension function such that $r^{-n}g(r)\to\infty$ as $r\to 0.$ Assume that $r\mapsto r^{-n}g(r)$ is decreasing and $r\mapsto r^{1-n}g(r)$ is increasing.  Then
\begin{equation*}
  \HH^g( \DDD)=\left\{\begin{array}{cl}
 0 &  {\rm if } \quad\sum\limits_{\qq\in\Z^n\setminus \{\0\}}\|\qq\|^n\Psi(\|\qq\|)^{1-n}g\left(\frac{\Psi(\|\qq\|)}{\|\qq\|}\right)< \infty.\\[3ex]
 \infty &  {\rm if } \quad \sum\limits_{\qq\in\Z^n\setminus \{\0\}}\|\qq\|^n\Psi(\|\qq\|)^{1-n}g\left(\frac{\Psi(\|\qq\|)}{\|\qq\|}\right)=\infty \text{ and $\Psi$ is decreasing or $n\geq 2$}.
                                     \end{array}\right.
\end{equation*}
\end{theorem}


  The conditions on the approximating function and dimension function  comes into play only for the divergence case.  In particular, the assumption that $\Psi$ is decreasing for $n=1$ is necessary due to work of Duffin--Schaeffer, as indicated above.   For $n\geq2$, the monotonicity assumption on the approximating function $\Psi$ is not required  \cite{BeresnevichVelani4}.

\subsection{Dual Diophantine approximation on manifolds}

Diophantine approximation on manifolds (or Diophantine approximation of dependent quantities) concerns the study of approximation properties of points in $\R^n$ which are functionally related or in other words restricted to a sub-manifold $\MM$ of $\R^n$. To estimate the size of sets of points $\xx\in\R^n$ which lie on a $k$-dimensional, nondegenerate\footnote{In this context `nondegenerate' means suitably curved, see \cite{Beresnevich3,KleinbockMargulis2} for precise formulations. Actually, within the theory of metric Diophantine approximation on manifolds, most of the results requires the manifolds to be nondegenerate. Removing this assumption presents significant challenges and the progress thus far has been achieved for limited  situations only, see for example \cite{BGGV} for the latest results in this direction.}, analytic submanifold $\MM\subseteq \R^n$ is an intricate and challenging problem. The fundamental aim is to estimate the size of the set $\MM\cap\DDD$ in terms of Lebesgue measure, Hausdorff measure and Hausdorff dimension. When asking such questions it is natural to phrase them in terms of a suitable measure supported on the manifold, since when $k<n$ the $n$-dimensional Lebesgue measure of $\MM\cap \DDD$ is zero irrespective of the approximating functions. For this reason,  results in the dependent Lebesgue measure theory (for example, Khintchine--Groshev type theorems for manifolds) are posed in terms of the $k$-dimensional 
Lebesgue measure (=Hausdorff measure) on $\MM$.

 

In full generality, a complete Hausdorff measure treatment akin to Theorem \ref{JSDV} for manifolds $\MM$  represents a deep open problem in the theory of Diophantine approximation. The problem is referred to as the Generalised Baker-Schmidt Problem (GBSP) inspired by the pioneering work of Baker-Schmidt \cite{BakerSchmidt}.  Ideally one would want to solve the following problem in full generality.

\begin{problem}[Generalised Baker-Schmidt Problem for Hausdorff Measure] Let $\MM$ be a nondegenerate submanifold of $\R^n$ with $\dim\MM=k$ and $n \geq 2$. Let $\Psi$ be an approximating function.  Let $g$ be a dimension function such that $r^{-k}g(r)\to\infty$ as $r\to 0.$ Assume that $r\mapsto r^{-k}g(r)$ is decreasing and $r\mapsto r^{1-k}g(r)$ is increasing. Prove that 

\begin{equation*}
  \HH^g( \DDD\cap\MM)=\left\{\begin{array}{cl}
 0, &  {\rm if } \quad\sum\limits_{\qq\in\Z^n\setminus \{\0\}}\|\qq\|^k\Psi(\|\qq\|)^{1-k }g\left(\frac{\Psi(\|\qq\|)}{\|\qq\|}\right)< \infty.\\[3ex]
 \infty, &  {\rm if } \quad \sum\limits_{\qq\in\Z^n\setminus \{\0\}}\|\qq\|^k\Psi(\|\qq\|)^{1-k}g\left(\frac{\Psi(\|\qq\|)}{\|\qq\|}\right)=\infty.
                                     \end{array}\right.
\end{equation*}
\end{problem}

In fact, this problem is stated in an idealistic format and solving it in this form is extremely challenging. The main difficulties lie
 in the convergence case and therein constructing a suitable nice cover for the set $ \DDD\cap\MM$. We list the contributions to date to highlight the significance of this problem.



Diophantine approximation on manifolds dates back to the profound conjecture of K. Mahler \cite{Mahler2} in 1932, which can be rephrased as the statement that $\HH^1(\DDh\cap \VV_n)=0$ for $\Psi(q)=q^{-\tau}$ with $\tau>n$. Here and throughout $\VV_n=\{(x, x^2, \ldots, x^n) : x\in\R\}$ denotes the \emph{Veronese} curve. Mahler's conjecture was eventually proven in 1965 by Sprind\v zuk \cite{Sprindzuk}, who then conjectured that the same holds when $\VV_n$ is replaced by any nondegenerate analytic manifold, and $\HH^1$ is replaced by $\HH^k$ where $k$ is the dimension of this manifold. Although particular cases of Sprind\v zuk's conjecture were known, it was not until 1998 that Kleinbock and 
   Margulis \cite{KleinbockMargulis2} established Sprind\v zuk's conjecture in full generality by using dynamical tools based on diagonal flows on homogeneous spaces. This breakthrough result acted as a catalyst for the subsequent progress in this area of research. The convergence Lebesgue measure result for $\DDh\cap\MM$ was established first in \cite{Beresnevich3} and then in a generalised form in \cite{BKM2}. The divergence case was established for the Veronese curve in \cite{Ber99} and for any nondegenerate manifolds in  \cite{BBKM}.  
  The first inhomogeneous Lebesgue measure result for the divergence case of $\DDD\cap\MM$ was established very recently in \cite{BBV}.

With regards to the Hausdorff measure or dimension theory on manifolds,   
  the progress has proven to be  difficult.   The first major result in this direction relating to the approximation of points on the Veronese curve appeared in the landmark paper of Baker and Schmidt \cite{BakerSchmidt} in which they proved upper and lower bounds for Hausdorff dimension on the Veronese curve. Further, they conjectured that their lower bound is actually sharp,
which was later proven to be true in \cite{Bernik}.
In 2000,  Dickinson and Dodson \cite{DickinsonDodson2} proved a lower bound for Hausdorff dimension on extremal manifolds for approximating functions 
of the type $\Psi(q)=q^{-\tau}$ for any $\tau>n$.
In 2006, the divergence case for the $\HH^g$-measure of $\DDh\cap\MM$ was established in \cite{BDV} as a consequence of their ubiquity framework.


 Proving the genuine Hausdorff measure result for convergence for any manifold remained out of reach until recently when the first-named author proved it  in \cite{Hussain} for $\DDp\cap \VV_2$ i.e. for $\HH^g$-measure on the parabola under some mild assumptions on the dimension function $g$. Soon after that, J.-J. Huang \cite{Huang}  proved that $\HH^s(\DDp\cap\CC)=0$ for all nondegenerate planar curves $\CC$ if a certain sum converges. 

Inhomogeneous Diophantine approximation is the generalisation of the homogeneous Diophantine approximation discussed above and results for the latter settings  can be derived from the former settings by considering the inhomogeneous parameter to be zero. We refer the reader to \cite{BBV, BHH, HSS} for accounts of inhomogeneous Baker-Schmidt problem. In particular,  authors of this paper proved the GBSP for hypersurfaces 
for $\HH^g$-measure upon a mild assumption on $g$, for both homogeneous and inhomogeneous settings with non-monotonic multivariable approximating functions \cite{HSS}.  However, the results of \cite{HSS} are not applicable to the one dimensional curves, for example,  planar curves or Veronese curves.  
%
%

 The aim of this paper is to contribute in making some advances on the GBSP on nondegenerate curves in a reasonable generality. We emphasize on the case of Veronese curves, especially the parabola.  
 In Section~\ref{rempc} we discuss planar curves.

\section{The generalised Hausdorff measure on the Veronese curve} \label{se2}
 We start with the simple case of a parabola to familiarise with some intricacies of reducible and  irreducible polynomials. 


\subsection{The parabola} \label{theparabola}Throughout, denote by $$\VV_2=\{(x,x^{2})\in\mathbb{R}^{2}: x\in[0,1]\}$$ the standard parabola. 
  In this case, we change the notation slightly,  instead of $\qq\cdot (x,x^2)-p$ we use  $a_{2}x^2+a_1x+a_0$, i.e. $(\qq, p) \to (a_2, a_1, a_0)$ and  $\|\aa\|=\max\{ |a_{1}|,|a_{2}|\}$.  So we consider the set  
 \begin{equation*}
  \DD(\VV_2, \Psi)= \DD_2(\Psi) \cap  \VV_2=
  \left\{x\in[0,1]: \begin{array}{l}
  |a_2x^2+a_1x+a_0|<\Psi(\|\mathbf a\|) \\
  \text{for\  i.m.}\ (a_0, a_1, a_2)\in\Z^3
                                  \end{array}\right\}.
\end{equation*}

As stated earlier, the first result regarding the convergence case of the $\HH^g$-measure on the parabola was proven  by Hussain \cite{Hussain}.

\begin{theorem}[Hussain, 2015]\label{Hussainparabola}
Let $\Psi$ be an approximating function and let $g$ be a dimension function such that $r^{-1}g(r)\to \infty$ as $r\to 0$
and $g(r)r^{-1}$ is decreasing. 
Assume that there exist positive constants $s_1$ and $s_2\leq 1$ such 
that $2s_1<3s_2$ and
\begin{equation}\label{zzz}
    r^{s_1}< g(r) < r^{s_2}\qquad\text{for all sufficiently small $r>0$}\,.
\end{equation}
Then
\begin{equation}\label{conresult}\HH^g( \DD(\VV_2, \Psi))=
  0 \ \  {\rm if } \quad \sum\limits_{q=1}^{\infty}g\left(\frac{\Psi(q)}{q}\right)q^2<\infty. \end{equation}
 

\end{theorem}


 Clearly one can see that Theorem  \ref{Hussainparabola}  is not as general as the convergence part of the GBSP for Hausdorff measure.
It is desirable to drop the condition \eqref{zzz}.
As an intermediate step, we introduce alternative conditions,
on the approximating function $\Psi$ and the dimension function
$g$, which allow the conclusion in \eqref{conresult}.
Our first result deals with a special class of approximating
functions of fast decay.

\begin{theorem} \label{thm2}
Let $\Psi$ be an approximating function such that for $q\geq q_0$, we have
\begin{equation} \label{bed}
\Psi(q^{2})\leq \Psi(q)^{2}.
\end{equation}
Then for any dimension function $g$ with the property that $r^{-1}g(r)\to \infty$ as $r\to 0$, and $r\mapsto r^{-1}g(r)$ decreasing, we have the conclusion
\eqref{conresult} of Theorem~\ref{Hussainparabola}. 
\end{theorem}
%
Observe that there is an identity in \eqref{bed} for any
 function $\Psi: q\mapsto q^{-\tau}$. Moreover
the condition is easy to check for functions of exponential decay.
We want to explicitly highlight our result for these natural classes of approximating functions.

\begin{corollary} \label{cor1}
Let $\tau>0$. Let $\Psi$ be of the 
form $\Psi: q\mapsto q^{-\tau}$ or of the form
$\Psi: q\mapsto e^{-q\tau}$. Let $g$ be any dimension function such that $r^{-1}g(r)\to \infty$ as $r\to 0$ and $r\mapsto r^{-1}g(r)$ is decreasing. Then the conclusion \eqref{conresult} holds. 
\end{corollary}

The proof of Theorem \ref{Hussainparabola} splits into two parts: the polynomial $P(x)=a_2x^2+a_1x+a_0$ having repeated roots or distinct roots. To be precise, identifying $P$ with its coefficient vector $\aa$,
the limsup set $\mathcal D(\VV_2, \Psi)$ can be covered by
\begin{align*}
\mathcal D(\VV_2, \Psi)=\mathcal D_{(1)}(\VV_2, \Psi)\cup \mathcal D_{(2)}(\VV_2, \Psi)
\end{align*}
where the index brackets are used to distinguish them from the sets
$D_n(\Psi)$ above and we have put
 \begin{equation*}\label{rr}
 \mathcal D_{(1)}(\VV_2, \Psi)=\left\{ \bigcap_{N\geq 1} \bigcup_{k\geq N}^\infty{}\bigcup_{2^k\leq\|\aa\|<2^{k+1}}\Delta(P): P {\rm \  has \ repeated \  roots}\right\},
\end{equation*}
$$ \mathcal D_{(2)}(\VV_2, \Psi)=\left\{ \bigcap_{N\geq 1}  \bigcup_{k\geq N}^\infty{}\bigcup_{2^k\leq\|\aa\|<2^{k+1}}\Delta(P): P {\rm \  has \ distinct \  roots}\right\},$$
and
\begin{equation*}
  \Delta(P)=\{x\in [0, 1): |P(x)|=|a_2x^2+a_1x+a_0|<\Psi(\|\aa\|)\}.
\end{equation*}

Hence, the desired statement that $\mathcal H^g(\mathcal D(\VV_2, \Psi))=0$ follows by establishing separately, $\mathcal H^g(\mathcal D_{(1)}(\VV_2, \Psi))=0$ and $\mathcal H^g(\mathcal D_{(2)}(\VV_2, \Psi))=0$.  In the paper, \cite[pp.55--57]{Hussain}, the proof of the distinct root case has been provided without assuming the assumption \eqref{zzz} on the dimension functions $g$.  Condition \eqref{zzz} is only needed for the repeated roots case (see the proof on \cite[p.54]{Hussain}). In the proof of Theorem \ref{thm2}, we provide a modified treatment (to the corresponding proof for the repeated root case \cite[p.54]{Hussain}) and as a consequence we obtain 
\begin{equation*}
\mathcal H^g(\mathcal D_{(1)}(\VV_2, \Psi))=0 \quad {\rm if} \quad \sum_{k=1}^{\infty} 2^{2k} g\left(\frac{\sqrt{\Psi(2^{2k})}}{2^{k}}\right)<\infty.
\end{equation*} 
%
%
%
On the other hand, the series in Theorem \ref{Hussainparabola}, arising from covering the set $\mathcal D_{(2)}(\VV_2, \Psi)$, is \[
\sum_{q=1}^{\infty} q^{2}g\left(\frac{\Psi(q)}{q}\right)\asymp \sum_{k=1}^{\infty} 2^{3k} g\left(\frac{\Psi(2^k)}{2^k}\right).\]

Hereby we use our assumption
that $\Psi$ is decreasing to see that the parameter change from $q$ to $k$
is justified, upon a suitable fixed factor. 
This leads us to study the 
convergence/divergence relations
between the two series
\begin{equation} \label{series1}
\sum_{k=1}^{\infty} 2^{3k} g\left(\frac{\Psi(2^k)}{2^k}\right)
\end{equation}
and
\begin{equation} \label{series3}
\sum_{k=1}^{\infty} 2^{2k} g\left(\frac{\sqrt{\Psi(2^{2k})}}{2^{k}}\right).
\end{equation}

%

We show a result somehow dual to Theorem~\ref{Hussainparabola}, where
we allow arbitrary dimension function $g$ at the cost
of a condition on $\Psi$ similar to \eqref{zzz}.

\begin{theorem} \label{1thm}
Let $s_{1},s_{2}$ be positive real numbers and suppose that
\[
q^{-s_2} \leq \Psi(q) \leq q^{-s_1}
\]
where $2s_2 \leq 3s_1 + 1$. Then, for any dimension function $g$, if \eqref{series1} converges, the series \eqref{series3} converges as well.
\end{theorem}

 On the other hand 
we have the following dichotomy.

\begin{theorem} \label{2thm}
There exist pairs of a (decreasing) approximating function $\Psi$
and a dimension function $g$ such that $r^{-1}g(r)\to \infty$ as $r\to 0$, and $r\mapsto r^{-1}g(r)$ is decreasing, and so that
\eqref{series1} converges but \eqref{series3}
diverges.
\end{theorem}

\begin{corollary}\label{korola}
Let $g$ and $\Psi$ be as 
in Theorem~\ref{2thm}. Then,  we have
\[
\HH^g(\DD( \VV_2, \Psi)) = \infty.
\]
\end{corollary}

\subsection{General Veronese curves and irreducibility} \label{gvc}
As explained above, in the proof of Theorem~\ref{Hussainparabola}, the restriction \eqref{zzz} was only required for the
class of reducible polynomials. The proof of Theorem~\ref{thm2} relies on the condition \eqref{bed} on $\Psi$ which can be regarded as a replacement for the condition \eqref{zzz} on $g$ in Theorem \ref{Hussainparabola}.

We now consider the Veronese curve in
arbitrary dimension $n$.
In view of the above observation, one may expect 
that, in general, the reducibility of polynomials 
might cause problems. However,
generalizing the ideas of Theorem~\ref{thm2}, our new results
(Theorem~\ref{thm4} and Theorem~\ref{abc}) below suggest that reducible
polynomials should not play an important role
upon some
strengthened version of \eqref{bed} on the approximating 
function $\Psi$. 

 We denote by 
$H_{P}=\max_{0\leq i\leq n} \vert a_{i}\vert$ the height of the polynomial
$P(t)=a_{n}t^{n}+\cdots+a_{1}t+a_{0}$.  The careful reader may observe that,  strictly speaking, this
treats a slightly altered problem than introduced 
in Section~\ref{intro}, where the approximating function only
depends on $n$ instead of $n+1$ elements (not on $a_{0}$). 
However, for fixed $t=\xi$,  by triangle inequality $|a_{0}|\lessless_{n,\xi} \max\{|a_{1}|,\ldots,|a_{n}|\}$ if $|P(\xi)|$ is small, and the results of this section can be verified similarly in the classical setting as well.

\begin{theorem}\label{thm4}
	Let $n\geq 1$ be an integer.
Let $\Psi$ be an approximating function with the property that for 
large $q>q_0$
the derived function 
\begin{equation}  \label{tauu}
q\longmapsto \frac{\log \Psi(q)}{\log q}
\end{equation}
is monotonically nonincreasing.
Let $x$ be a real number not algebraic of degree at most $n$.
Suppose that the set of polynomials
\begin{equation*} \label{ihredu}
Z_{\Psi}^{\prime}(x):= \{ P(t)= a_{n}t^{n}+\cdots+a_{0}\in \Z[t]:\; \vert P(x)\vert\leq \Psi(H_{P})\}
\end{equation*}
is infinite.
Assume that for every $c_{1}>1$ there
exists $c_{2}>0$ such that
\begin{equation} \label{gleichung}
\Psi(q/c_{1})\leq c_{2}\Psi(q), \qquad\qquad q\geq q_{0}(c_{1}).
\end{equation}
Then there exists an explicitly computable constant $C=C(n,x)>0$ and
infinitely many irreducible integer polynomials $Q$
of degree at most $n$ with the property
\begin{equation} \label{eqq}
\vert Q(x)\vert\leq C\Psi(H_{Q}).
\end{equation}
These polynomials $Q$ can be chosen among divisors of
polynomials in $Z_{\Psi}^{\prime}(x)$.
\end{theorem}

\begin{remark}  \label{R1}
	In the arXiv version (first and second version) of this paper~\cite{arXiv}, we additionally 
	show the following claim that admits a similar but 
	slightly more technical proof.
	Assume there are infinitely many 
	polynomials $P\in Z_{\Psi}^{\prime}(x)$
	with the property that when factorized into irreducible polynomials
	\[
	P(t)=Q_{1}(t)Q_{2}(t)\cdots Q_{k}(t), \qquad \qquad k=k(P)\leq n,
	\]
	then there is no index $i\in\{1,2,\ldots,k\}$ so that
	\begin{equation} \label{gl22}
	H_{Q_{i}}> H_{P}.
	\end{equation}
	Then, for any $\epsilon>0$,
	there exist infinitely many irreducible integer polynomials $Q$ (among
	the $Q_{i}$ above) such that the estimate
	\begin{equation}  \label{gl3}
	\vert Q(x)\vert\leq \Psi(H_{Q})^{1-\epsilon}
	\end{equation}
	holds. Moreover, if an index $i$ satisfies \eqref{gl22} then
	for the codivisor polynomial we have
	\begin{equation} \label{eq:mala}
	H_{P/Q_{i}}\leq \binom{n}{\left\lfloor n/2\right\rfloor}\sqrt{n+1}.
	\end{equation}
	Hence \eqref{gl3} is in particular true if 
	the following property $(\ast)$ holds:
	\smallskip
		
		\noindent{\bf Property ($\ast$).} Only finitely many $P\in Z_{\Psi}^{\prime}(x)$
		have a fixed, non-constant divisor polynomial. 
		\smallskip
		
		This property further implies
		that $C=C(n)$ can be chosen independently from $x$ in 
		Theorem~\ref{thm4}. 
%
	\end{remark}

\begin{remark}  
	The assumption on $x$ is necessary. If $x$ is algebraic of degree at most 
	$n$ and	we take $\Psi(q)=q^{-n}$ (or
	any other $\Psi(q)=o(q^{-n+1})$), a variant of 
	Liouville's inequality~\cite[Theorem~A1]{Bugeaud} implies that 
	$Z_{\Psi}^{\prime}(x)$ consists essentially only
	of the infinite set
	of polynomial multiples of  
	its minimal polynomial not exceeding degree $n$ (including scalar multiples),  up to finitely many exceptions.
	Thus for any $C>0$ there are only
	finitely many solutions to \eqref{eqq} in irreducible polynomials.  
\end{remark}

\begin{remark}
	Notice that in the nontrivial case
	when the polynomials $P$ are reducible, then the derived
	$Q$, in fact, have degree at most $n-1$. Among such polynomials we
	should generically expect a significantly weaker order of approximation
	than for polynomials of degree $n$. 
\end{remark}

Condition \eqref{gleichung} controls the decay of $\Psi$ within short intervals 
(on a logarithmic scale). Without loss of generality we may only impose
\eqref{gleichung} for $c_{1}=2$. Write $c_2=c_2(c_1)$.
Indeed, since $\Psi$ decreases, if
$K\geq 1$ is the integer so that $2^{K-1}<c_1\leq 2^{K}$, 
then repeated application with $c_1=2$ 
shows that we may take $c_2(c_1)= c_2(2)^{K}$.  
\smallskip

Theorem~\ref{thm4} tells us that, for such $\Psi$,
we may restrict to $\Psi$-approximable
{\em irreducible} polynomials, at the cost of a minor 
modification of the involved 
approximating function $\Psi$. The twist by a constant $C$ 
as in \eqref{eqq} is probably the best we can hope for.
For a wide class of dimension functions $g$, the convergence of the series \eqref{conresult} will remain unaffected upon the replacement of $\Psi$ by $C\Psi$
as in \eqref{eqq}. In particular this is true when $r\mapsto r^{-1}g(r)$ is decreasing as in
the hypothesis of the GBSP. See also the derived metric result Theorem~\ref{abc} and its proof below.
 For the purpose of this result we only require the first claim of Theorem~\ref{thm4},
the second is added more for sake of completeness.
The conditions \eqref{tauu}, \eqref{gleichung}
apply to power  functions $\Psi: q\mapsto q^{-\tau}$ (with
$c_{2}= c_{1}^{\tau}$ in \eqref{gleichung}), 
as well as to several 
variations. An interesting class of functions
is provided in the next corollary. 


\begin{corollary} \label{obn}
Let $N\geq 0$ be an integer and let
		\[
		\Psi(q)= q^{-\tau}\cdot \varphi_{1}(q)\varphi_{2}(q)\cdots \varphi_{N}(q)
		\]
		 where $\tau>0$ and $\varphi_{i}:(0,\infty)\to (0,\infty)$, $1\leq i\leq N$, are any functions that satisfy
	\begin{itemize}
	 \item $\varphi_{i}(q)>1$ for $q\geq q_{0}$
	 \item either $\varphi_{i}$ is decreasing for $q\geq q_{0}$, or 
	  $\varphi_{i}$ is increasing and differentiable for $q\geq q_{0}$ and \[
	  \frac{d}{dq}\log(\varphi_{i}(q))=o((q\log q)^{-1}),\qquad q\to\infty.
	  \]
	\end{itemize}
	Then \eqref{tauu} and \eqref{gleichung}
	are both satisfied and consequently the first claim of Theorem~\ref{thm4} holds. 
\end{corollary}

If $\varphi=\varphi_{i}$ increases,
the last condition of the corollary means, roughly speaking, that
it grows slower than the logarithm. 
The second condition can be slightly relaxed upon assuming that
$\varphi_i$ increases to infinity. Then the assumption
$\frac{d}{dq}\log(\varphi_{i}(q))=O((q\log q)^{-1})$ with an arbitrary implied constant suffices.
In case $\varphi_i$ increases to some finite limit $L\geq 1$, then 
suitable constants may depend on $L$.

The next example contains another class
of functions applicable for Theorem~\ref{thm4},
although they do not necessarily satisfy the assumptions of Corollary~\ref{obn}.

\begin{example} \label{exa}
Let $\Psi$ be a function given for large $q\geq q_{0}$ 
 as
\[
\Psi(q)=Aq^{-\tau}(\log q)^{a}(\log \log q)^{b} (1+e^{-cq}),
\]
for $A>0, \tau>0$ and
$a\geq 0,b\in\mathbb{R}, c\geq 0$, where we require 
$b\geq 0$ if $a=0$. Then $\Psi$ satisfies the conditions 
\eqref{tauu} and \eqref{gleichung} of Theorem~\ref{thm4}.
\end{example}
Unfortunately, condition \eqref{gleichung} no
longer applies to functions with rapid decay, 
such as $\Psi(q)=e^{-\kappa q}$ for $\kappa>0$.
The biased splitting condition \eqref{gl22} in the second claim 
enters for technical reasons,
and we strongly believe it is not needed for the conclusion
of \eqref{gl3}. It appears to be very weak, however we are unable to get rid
of it at present. 
The condition \eqref{gl22} is particularly satisfied if $P$ is a 
power of an irreducible polynomial, which was the case that caused most problems
in the proof of Theorem~\ref{Hussainparabola}. Also it holds when there
exists no polynomial $Q(t)$ as in the problem below.

\begin{problem}
	Is there an integer polynomial $Q(t)$ so that for infinitely many 
	integer polynomials $R(t)$ of bounded degree $\leq n$
	with coprime coefficient vector\footnote{By the coprime coefficient vector we mean that the gcd of coefficients is $1$.},
	we have $H_{RQ}<H_{R}$? 
	\end{problem}


 Unfortunately, the answer is positive when not restricting the degree of $R$,
as we may take $Q$ a divisor of a cyclotomic polynomial
of height greater than $1$, for the existence see for 
example~\cite{Moree}.

We present a variant of Theorem~\ref{thm4} where the problems
regarding the implied constants mentioned in the paragraph above Corollary~\ref{obn} can be avoided, using the multiplicative Mahler measure instead of the standard height. In particular in this setting condition
\eqref{gleichung} is no longer needed.
We recall that for a complex polynomial $P(t)=a\prod (t-\alpha_{i})$, its Mahler measure $M_{P}$ is given as  $M_{P}=|a| \prod \max\{ 1,|\alpha_{i}|\}$.

\begin{theorem} \label{weil} Let $x$ be real, not algebraic of degree at most $n$. Let $\Psi$ be an approximating function that satisfies \eqref{tauu}.
	Then assume that the set
	\[
	\mathscr{Z}_{\Psi}^{\prime}(x):= \{ P(t)= a_{n}t^{n}+\cdots+a_{0}\in \Z[t]:\; \vert P(x)\vert\leq \Psi(M_{P})\}
	\] 
	is infinite. Then there exists an explicitly computable 
	constant $C=C(n,x)>0$ and
	infinitely many irreducible integer polynomials $Q$
	of degree at most $n$ with the property
	\[
	\vert Q(x)\vert\leq C\Psi(M_{Q}).
	\]
	If $\mathscr{Z}_{\Psi}^{\prime}(x)$ satisfies property $(\ast)$
	from Remark~\ref{R1}, we may let $C=1$.
\end{theorem}
 
 Notice that $M_{P}\asymp_{n} H_{P}$, see~\cite{Mahler3} (or the 
 proof of Theorem~\ref{thm4} below).
 While we believe \eqref{gleichung} resp. \eqref{gl22} are not required for the respective claims of Theorem~\ref{thm4} either, the assumption \eqref{tauu} seems to be crucial in both
Theorem~\ref{thm4} and Theorem~\ref{weil}.
The next theorem illustrates that we cannot drop both \eqref{tauu} and \eqref{gleichung} 
completely. 

\begin{theorem} \label{bugresu}
  Let $n\geq 2$ be an integer. There exists a transcendental
  real number $x$, a nonincreasing approximating function $\Psi(q)$ that
  tends to $0$ as $q\to\infty$ and a positive constant $\Delta(n)$
  with the following properties.
  \begin{itemize} 
  	\item The set of integer polynomials 
  	\[
  	Z_{\Psi}^{\prime}(x)=\{  P(t)= a_{n}t^{n}+\cdots+a_{0}\in \Z[t]:\; \vert P(x)\vert\leq \Psi(H_{P})\}
  	\]
  	%
  	is infinite.
  \item For any irreducible integer polynomial $Q(t)$ of 
  degree at most $n$ the estimate
  \[
  \vert Q(x)\vert > \Delta(n)\cdot \Psi(H_{Q})^{1/n}
  \]
  holds. In particular, \eqref{gl3} fails for any $\epsilon<1/2$
  (consequently \eqref{eqq} fails as well).
\end{itemize}
\end{theorem}

   Observe that $\Psi(H_{Q})^{1/n}$ increases with $n$, 
   hence the second claim becomes stronger
   as $n$ grows.  In fact, we construct a set of $x$, of positive Hausdorff dimension $1/(2n^2-n+1)$,  which satisfy the claims of Theorem~\ref{bugresu} for
   some associated function $\Psi=\Psi_{x}$ within the proof of Theorem \ref{bugresu}.
   It may even be feasible to take a common $\Psi$ independent of $x$. We continue the discussion of this issue  in
   Section~\ref{bgrs}.
  
  We present an application of our results on irreducibility, stated
  above,  that relates to the work of Pezzoni~\cite{Pezzoni}.
	To formulate claims in~\cite{Pezzoni}, we introduce some notation first. 
	Following Pezzoni, for given $c>0$,
	let $P_{n,\lambda}$ be the set of polynomials of degree
	at most $n$ with discriminant $D(P)$ bounded by
	\begin{equation}  \label{eq:disco}
	\vert D(P)\vert\leq cH_{P}^{2(n-1-\lambda)}
	\end{equation}
	and $P_{n,\lambda}^{\ast}$ the same set but restricted
	to irreducible polynomials $P$. Furthermore let
	\[
	A_{n,\lambda}(\Psi)= \{ x\in \R: \vert P(x)\vert \leq \Psi(H_P) 
	\quad \text{for i.m.} \; P\in P_{n,\lambda}\},
	\]
	and likewise define the subset $A_{n,\lambda}^{\ast}(\Psi)$ by
	considering only $P\in P_{n,\lambda}^{\ast}$.
	Thereby, if $\lambda>0$, 
	we restrict the approximation
	to those polynomials with relatively small discriminant
	(note that for generic polynomials we expect $D(P)$ to be of size roughly
	$H_{P}^{2n-2}$). 
	From now on, we deal with $n=3$ only.
	
	\begin{theorem}[Pezzoni] \label{PEZZ}
		Let $\Psi$ be an approximating function. Let $g$ be a dimension function such that $r\mapsto r^{-1}g(r)$ is decreasing
		and $r^{-1}g(r)\to \infty$ as $r\to 0$. Let $0\leq \lambda<9/20$.
		Then
		\begin{equation} \label{eq:PEZ1}
		\HH^{g}(A_{3,\lambda}^{\ast}(\Psi))=0, \qquad \text{if} \quad
		\sum_{q=1}^{\infty} g\left(\frac{\Psi(q)}{q}\right)q^{3-2\lambda/3}<\infty.
		\end{equation}
		If $\Psi(q)=q^{-\tau}$,  for some $\tau>0$, we denote $A_{3,\lambda}(\tau)=A_{3,\lambda}(\Psi)$. Then,
		\begin{equation}  \label{eq:PEZ2}
			\HH^{g}(A_{3,\lambda}(\tau))=0, \qquad \text{if} \quad
		\sum_{q=1}^{\infty} g(q^{-\tau-1})q^{3-2\lambda/3}<\infty.
		\end{equation}
	\end{theorem}

   As noticed in~\cite{Pezzoni}, the choice of $c$ in \eqref{eq:disco} is irrelevant for the claim.
   The step to deduce \eqref{eq:PEZ2}, i.e~\cite[Corollary 1.6]{Pezzoni}
   from \eqref{eq:PEZ1}, that is, \cite[Theorem~1.5]{Pezzoni}, uses the property that polynomial heights behave
   almost multiplicatively (up to a factor depending on 
   the degrees only),
   which is a crucial point in the proof of Theorem~\ref{thm4} as well. 
   As an application of Theorem~\ref{thm4}, 
   we generalize \eqref{eq:PEZ1} to the sets $A_{3,\lambda}(\Psi)$ 
   without restriction to irreducible polynomials,
   which in turn also
   generalizes \eqref{eq:PEZ2} to a wider class 
   of approximating functions. For $\lambda=0$,
   this in particular extends
   Bernik's result~\cite{Bernik} for $n=3$, i.e. $\dim_{\HH}(A_{3,0}(\tau))=4/(\tau+1)$
   where $\dim_{\HH}$ denotes Hausdorff dimension, 
   to more general approximation functions and to the Hausdorff measure which is obviously stronger than the Hausdorff dimension\footnote{To be precise, as a consequence of our result (Theorem \ref{abc}), we obtain an upper bound on the Hausdorff dimension of $A_{3,0}(\tau)$. The lower bound follows from the divergence result proved in \cite{BDV}.}.  

\begin{theorem} \label{abc}
	Suppose $\Psi$ is an approximating function that satisfies
	\eqref{tauu} and \eqref{gleichung}. Let $g$ be a dimension function so that
	$r\mapsto r^{-1}g(r)$ decreases and $r^{-1}g(r)\to \infty$ as $r\to 0$.  We have
	%
	\[
	\HH^{g}(A_{3,\lambda}(\Psi))=0, \qquad \text{if} \quad
	\sum_{q=1}^{\infty} g\left(\frac{\Psi(q)}{q}\right)q^{3-2\lambda/3}<\infty.
	\]
\end{theorem}

The claim extends \eqref{eq:PEZ1} to the larger set $A_{3,\lambda}(\Psi)$
in place of $A_{3,\lambda}^{\ast}(\Psi)$, and is therefore stronger
and cleaner, upon assuming \eqref{tauu} and \eqref{gleichung}. 
Compared to \eqref{eq:PEZ2}, our result is
valid for a wider class of approximating functions that satisfy
our conditions \eqref{tauu} and \eqref{gleichung}, including all power functions.
Indeed conditions \eqref{tauu} and  \eqref{gleichung} allow for
providing the ``decoupling'' property indicated as an essential problem when going beyond multiplicative functions $\Psi$ in~\cite[Section~6]{Pezzoni}.
Our proof will show that similar transitions
from irreducible to general polynomials can be readily derived from Theorem~\ref{thm4} for
 any degree $n$.
 Similarly, a version of Theorem~\ref{abc}  where we take the
 Mahler measure as the height function without
 requiring the condition \eqref{gleichung} could be 
 readily obtained if the corresponding claim can be shown for irreducible polynomials.
 It seems that the method in~\cite{Pezzoni} can be extended to
  this Mahler measure setting, but we do not attempt to carry it out. 

\section{Remarks on GBSP on planar curves}  \label{rempc}


%

From the previous discussion, it should be clear that the main problem within the GBSP lies in proving the convergence part. The GBSP has been resolved over nondegenerate planar curves due to successive works of Hussain \cite{Hussain} and Huang \cite{Huang} albeit some restrictions on the dimension and  approximating functions.  Here we state the most general result due to  Huang \cite{Huang}.

\begin{theorem}[Huang, 2017]\label{BHH} Let $\Psi$ be a decreasing approximating function and $s\in~(0, 1]$. 
	Let $\CC$ be any  $C^2$ planar curve which is nondegenerate everywhere except possibly on a set of zero Hausdorff $s$-measure. Then
	
	\begin{equation}\label{Huangpla}\HH^s(\DDp\cap\CC)=0\quad {\rm  if}\quad \sum_{q=1}^\infty\Psi^s(q)q^{2-s}<\infty.\end{equation}
\end{theorem}
In a recent work, answering a question posed in \cite{Simmons10}, Huang \cite[Theorem 3]{Huang3} proved that the conclusion \eqref{Huangpla} holds without assuming monotonicity on the approximating function $\Psi$ for the parabola, that is for $\CC=\VV_2$. A natural intriguing question is that if the according conclusion holds if $\Psi$ is replaced by a multivariable approximating function,  $\psi:\Z^2\to[0, \infty)$ such that $\psi(\qq) \rightarrow~0 \text{ as } \|\qq\|:=\max(|q_1|,  |q_2|)\rightarrow~\infty.$
Hereby we sum over the non-zero integer points $\qq$ and therefore 
naturally expect that the exponent of $\qq$ drops by $1$ compared 
to \eqref{Huangpla}. However,
we show that it is not possible to remove the monotonicity assumption on the multivariable approximating function $\psi$. To be precise we show that the monotonicity assumption on a standard parabola $\VV_2=\{(x,x^{2})\in\mathbb{R}^{2}: x\in[0,1]\}$ is necessary.
Define $\DD_2(\psi)\cap\VV_2$ similarly as for 
single-variable approximation functions,
\begin{equation*}
\DD_2(\psi) \cap  \VV_2=
\left\{x\in[0,1]:   |q_1x+q_2x^2+p|<\psi(\qq)  \   \text{for i.m.} \ (p, q_1, q_2)\in\Z^{3}\setminus\{\0\}                               \right\}.
\end{equation*}

\begin{theorem}\label{Removingmono}There exists a non-monotonic multivariable approximating function $\psi$ such that for 
	any $s\in(0,1)$
	we have
	\begin{equation}\label{Multicounterexample}\HH^s(\DD_2(\psi)\cap\VV_2)=\infty\quad {\rm  but}\quad \sum_{\qq\in\Z^2\setminus\{\0\}}\psi^s(\qq)\|\qq\|^{1-s}<\infty.\end{equation}
	
\end{theorem}

We may extend the claim to $s=1$ upon replacing $\infty$ by $1$
in \eqref{Removingmono},
required by our definition of $\VV_2$.

\section{Hausdorff measure and dimension}\label{hm}

For completeness we give below a very brief introduction to Hausdorff measures and dimension. For further details see \cite{Falconer_book2013}.

Let
$\Omega\subset \R^n$.
 Then for any $0 < \rho \leq \infty$, any finite or countable collection~$\{B_i\}$ of subsets of $\R^n$ such that
$\Omega\subset \bigcup_i B_i$ and $\mathrm{diam} (B_i)\le \rho$ is called a \emph{$\rho$-cover} of $\Omega$.
Let
\[ 
\HH_{\rho}^{g}(\Omega)=\inf \sum_{i} g\left(\diam (B_i)\right),
\]
where the infimum is taken over all possible $\rho$-covers 
$\{B_i\}$ of $\Omega$. The \textit{$g$-dimensional Hausdorff measure of $\Omega$} is defined to be
\[
\HH^g(\Omega)=\lim_{\rho\to 0}\HH_\rho^g(\Omega).
\]
The map $\HH^g:\P(\R^n)\to [0,\infty]$ defines an outer measure
on all sets in $\mathbb{R}^n$, which becomes a proper 
measure when restricted to the subset of $\HH^g$-measurable sets,
i.e. sets $A$ that satisfy $\HH^g(B)= \HH^g(A\cap B)+\HH^g(B\setminus A)$ for
any $B\in \mathbb{R}^n$. 
In the case that $g(r)=r^s \;\; (s\geq 0)$, the measure $\HH^g$ is denoted $\HH^s$ and is called \emph{$s$-dimensional Hausdorff measure}. For any set $\Omega\subset \R^n$ one can easily verify that there exists a unique critical value of $s$ at which the function $s\mapsto\HH^s(\Omega)$ ``jumps'' from infinity to zero. The value taken by $s$ at this discontinuity is referred to as the \textit{Hausdorff dimension} of $\Omega$ is denoted by $\dim_{\HH} \Omega $; i.e.
\[
\dim_\HH \Omega :=\inf\{s\geq 0\;:\; \HH^s(\Omega)=0\}.
\]
The countable collection $\{B_i\}$ is called a \emph{fine cover} of $\Omega$ if for every $\rho>0$ it contains a subcollection that is a $\rho$-cover of $\Omega$.

%
%
We state the Hausdorff measure analogue of the famous Borel--Cantelli lemma (see \cite[Lemma 3.10]{BernikDodson}) which will allow us to estimate the Hausdorff measure of certain sets via calculating the Hausdorff $f$-sum of a fine cover.

\begin{lemma}[Hausdorff--Cantelli lemma]\label{bclem}
Let $\{B_i\}\subset\R^n$ be a fine cover of a set $\Omega$ and let $g$ be a dimension function such that
\begin{equation}
\label{fdimcost}
\sum_i g\left(\diam(B_i)\right) \, < \, \infty.
\end{equation}
Then $$\HH^g(\Omega)=0.$$
\end{lemma}
%
%


\section{Proofs of Theorems}


\subsection{Proof of Theorem~\ref{thm2}}

The key idea of the proof of Theorem~\ref{thm2} is a refined treatment of the case of polynomials
with a repeated root, where the additional condition \eqref{zzz}
was required in the original proof of Theorem~\ref{Hussainparabola}. Summarising the proof of Theorem \ref{Hussainparabola}, we recall that
\begin{align*}
\mathcal D(\VV_2, \Psi)= \mathcal D_{(1)}(\VV_2, \Psi)\cup \mathcal D_{(2)}(\VV_2, \Psi)
\end{align*}
with sets defined in Section~\ref{theparabola}.
Hence, the desired statement that $\mathcal H^g(\mathcal D(\VV_2, \Psi))=0$ follows by establishing separately,
\begin{center}
  \[ \mathbf{Case \ I:}\qquad \qquad\qquad \mathcal H^g(\mathcal D_{(1)}(\VV_2, \Psi))=0\]
  \[\mathbf{Case \ II:} \ \ \quad \qquad\qquad\mathcal H^g(\mathcal D_{(2)}(\VV_2, \Psi))=0.\]
\end{center}

As can be seen in the paper \cite[pp.55--57]{Hussain}, the proof of {\bf Case II} is clean and does not assume any restrictive assumption \eqref{zzz} on the dimension functions $g$. The assumption \ref{zzz} on the dimension function $g$  comes into play when the polynomial $P$ has repeated roots ({\bf Case I}). In this case we may write $P(x)=a_{2}\left(x-\frac{u}{v}\right)^{2}$,
so $|P(x)|$ being small corresponds to a rational number
$u/v$ being close to $x$. We can thus deduce
the claim from the one-dimensional convergence case of Theorem \ref{JSDV}.

%

%

For convenience let us rewrite the set $\mathcal D_{(1)}(\VV_2, \Psi)$ restricted to the repeated roots case. We have
%
\[
\DD_{(1)}(\VV_2,\Psi)=\left\{ x\in [0, 1): \exists \ \text{i.m.}\; P(x)= a_{2,P}x^{2}+a_{1,P}x+a_{0,P}= a_{2,P}\left(x-\frac{u_P}{v_P}\right)^{2}:\;
\vert P(x)\vert\leq \Psi(\|\aa_{P}\|)\right\},
\]
where for clarity we have indicated the dependency of $P$ in the index and
identified $P(t)=a_{2,P}t^2+ a_{1,P}t+a_{0,P}$ 
with its coefficient vector 
$\aa_{P}=(a_{2,P},a_{1,P},a_{0,P})$ of norm
$\|\aa_{P}\|= \max\{ |a_{1,P}|, |a_{2,P}| \}$. 
Clearly $a_{2,P}\neq 0$ unless $P\equiv 0$.
It is further understood that
we take the involved $u_P,v_P$ the unique induced coprime integer 
pair with $v_P>0$.
We want to estimate the Hausdorff measure of this set. 
Since $(u_P,v_P)=1$ and $P\in\mathbb{Z}[t]$ we have $v_P^{2}\vert a_{2,P}$
and thus $v_P^{2}\leq \vert a_{2,P}\vert$ since $a_{2,P}\neq 0$.

Again since $a_{2,P}\neq 0$, we may conclude
\begin{align*}
\DD_{(1)}(\VV_2,\Psi)&=  \left\{ x\in[0,1]: \exists \ \text{i.m.}\; P:
\vert x-u_P/v_P\vert \leq \sqrt{\Psi(\|\aa_P\|)}/\sqrt{a_{2,P}}\right\} \\
&\subseteq \left\{ x\in[0,1]: \exists \ \text{i.m.}\; P: \vert v_P\vert\leq \sqrt{\|\aa_P\|},\;
\vert x-u_P/v_P\vert \leq \sqrt{\Psi(\|\aa_P\|)}/v_P\right\}  \\
&= \left\{ x\in[0,1]: \exists \  \text{i.m.}\; P:
\vert v_P\vert\leq \sqrt{\|\aa_P\|},\; \vert v_Px-u_P\vert \leq \sqrt{\Psi(\|\aa_P\|)}\right\},
\end{align*}
since $v_P\leq \sqrt{\vert a_{2,P}\vert} \leq \sqrt{\|\aa_P\|}$. But
obviously $\vert u_P\vert < v_P$ as soon as $\|\aa_P\|$ is large enough since $x\in[0,1)$ and $\Psi$ tends to $0$.
Thus, we have
\[
\DD_{(1)}(\VV_2,\Psi)\subseteq B_{\Psi}:=
\left\{ x\in[0, 1): \exists \ \text{i.m. } P:
\max\{ \vert u_P\vert, \vert v_P\vert\}\leq \sqrt{\|\aa_P\|},\; \vert v_Px-u_P\vert \leq \sqrt{\Psi(\|\aa_P\|)}\right\}.
\]
Identifying $\sqrt{\|\aa_P\|}$ with $q$,
from the one-dimensional case of Theorem~\ref{JSDV} (note that
the convergence claim does not require monotonicity on $\Psi$)
it follows that
$$\HH^{g}(\DD_{(1)}(\VV_2,\Psi))\leq \HH^{g}(B_{\Psi})=0$$ as soon as $r^{-1}g(r)\to\infty$ as $r\to 0$,  $r\mapsto r^{-1}g(r)$ is decreasing,  and
\begin{equation} \label{eq:1}
\sum_{q=1}^{\infty} qg\left(\frac{\sqrt{\Psi(q^{2})}}{q}\right)<\infty.
\end{equation}
In order to deduce \eqref{eq:1} from the assumption
\begin{equation} \label{eq:2}
\sum_{q=1}^{\infty}q^2 g\left(\frac{\Psi(q)}{q}\right)<\infty, 
\end{equation}it suffices to find conditions on $\Psi, g$ such that \eqref{eq:2} implies \eqref{eq:1}. 
Obviously the criterion is met for any increasing
$g$ and $\Psi$ that satisfies \eqref{bed}.

\begin{remark} The reader might notice here that the argument presented above is modified from \cite[p.54]{Hussain} and hence the series \eqref{eq:1} corresponding to the repeated root case is slightly different from the corresponding series $\sum_{k=1}^{\infty} 2^{k} g\left(\sqrt{\Psi(2^{k})}/2^{k/2}\right)$  that arose from the corresponding proof in \cite{Hussain}. 
%
%
%
\end{remark}

\subsection{Proof of Theorem~\ref{1thm}}
Since $g$ is increasing, \eqref{series1} can be estimated as
\begin{equation} \label{betw}
\sum_{k=1}^{\infty} 2^{3k} g\left(\Psi(2^k)/2^k\right)\geq \sum_{k=1}^{\infty} 2^{3k} g\left(2^{-k(s_2 + 1)}\right)=: A_{g}.
\end{equation}
Similarly \eqref{series3} can be bounded from above by
\[
\sum_{k=1}^{\infty} 2^{2k} g\left(\frac{\sqrt{\Psi(2^{2k})}}{2^{k}}\right) \leq 
\sum_{k=1}^{\infty}  2^{2k} g\left(2^{-k(s_1 + 1)}\right)=:B_{g}.
\]
Let $\ell= s_1+ 1$.
Since $g$ increases, for every $k$ when considering three
consecutive terms indexed $3k,3k+1,3k+2$ in the sum $B_{g}$
we obtain
\[
2^{6k}g(2^{-3k\ell})+ 2^{6k+2}g(2^{-(3k+1)\ell})
+ 2^{6k+4}g(2^{-(3k+2)\ell})\leq (1+4+16)2^{6k}g(2^{-3k\ell})\]
and hence
\[
\sum_{k=3}^{\infty} 2^{2k} g(2^{-k\ell})\leq 21
\sum_{k=1}^{\infty} 2^{6k} g(2^{-3k\ell}).
\]
We see that the series $B_{g}$ converges
as soon as
\begin{equation} \label{between}
\sum_{k=1}^{\infty} 2^{6k} g(2^{-3k\ell})<\infty.
\end{equation}
Assume the left hand side of \eqref{betw} converges. This implies that the series $A_{g}$ converges, and in turn the convergence of the subseries $\sum_{k=1}^{\infty} 2^{6k} g(2^{-2k(s_2+1)})$ derived from restricting to even indices. Comparing the latter series with \eqref{between}, by using the monotonicity of $g$, it follows that the series $B_{g}$ converges as soon as $2(s_{2}+1)\leq 3\ell$, which is equivalent to the hypothesis on $s_{1},s_{2}$. Since  the series $B_{g}$ is bigger than \eqref{series3}, indeed the
convergence of \eqref{series1} implies the convergence
of \eqref{series3}.

\subsection{Proof of Theorem~\ref{2thm}}
Fix $\alpha > 3$ and $\beta>6$. 
Let $(Q_n)_{n\geq 1}$ be a sequence of integral powers
of $2$ of the form
\[
Q_n=2^{k_{n}}, \qquad Q_n > Q_{n - 1}^\alpha,
\]
that is, $k_{n}/k_{n-1}>3$. It will be convenient to let $k_0=0$. 
Let $\Psi$ be defined piecewise by the formula 
$\Psi(q) = Q_n^{2-2\beta}$ for all 
$Q_{n - 1}^2 < q \leq Q_n^2$. Clearly $\Psi$ is non-increasing.
Let $g(r)= \max_{n\geq 1} g_{n}(r)$ where
\[
g_n(r) = \begin{cases}
Q_n^{-2}, & r \geq Q_n^{-\beta}\\[2ex]
rQ_n^{\beta-2}, & r \leq Q_n^{-\beta}.
\end{cases}
\]
Notice that the values
$Q_n^{-\beta}=2^{-\beta k_{n}}$ decay as $n$ increases. Thus
the (right) neighborhood of $0$ where $g_{n}$ increases
become shorter whereas
the slopes of $g_{n}$ in these segments grow as $n$ increases.
Hence $g$ is piecewise linear with slopes among $\{0\}\cup \{ Q_{n}^{\beta-2}: n\geq 1\}$, and the smaller $r$
gets the larger the integer $n$ becomes for which $g(r)=g_{n}(r)$ holds locally. We further see that $g(r)/r$ is non-increasing. Indeed, 
monotonicity is easy to check for each $g_n(r)/r$ and therefore
it extends to the pointwise maximum $g(r)/r$. 
Moreover,
from the description of $g$ we readily verify that
\begin{equation} \label{erste}
g(Q_n^{-\beta}) = g_{n}(Q_n^{-\beta})= Q_n^{-2}, \qquad n\geq 1,
\end{equation}
which implies $g(r)/r\to\infty$ as $r\to 0$ by the aforementioned monotonicity and $\beta>2$, as well as
\begin{equation} \label{zweite}
g(Q_{n}^{-\gamma})\leq \max\{ Q_{n+1}^{-2},g_{n}(Q_{n}^{-\gamma})\}=
\max\left\{ Q_{n+1}^{-2}, Q_{n}^{\beta-2-\gamma}\right\}, \qquad \gamma\geq \beta.
\end{equation}
From \eqref{erste} we see that
\begin{align*}
\sum_{k=1}^{\infty} 2^{2k} g\left(\frac{\sqrt{\Psi(2^{2k})}}{2^{k}}\right) &\geq \sum_{n=1}^{\infty} 2^{2k_n} g\left(\frac{\sqrt{\Psi(2^{2k_n})}}{2^{k_n}}\right) \\ &= \sum_{n=1}^{\infty} Q_n^2 g\left(\frac{\sqrt{Q_n^{2-2\beta}}}{Q_n}\right) \\
 &= 
 \sum_{n=1}^{\infty} Q_n^2 g(Q_n^{-\beta}) \\ &= \sum_{n=1}^{\infty} Q_n^2 Q_n^{-2} \\ &= \infty.
\end{align*}
Thus \eqref{series3} diverges. We need to show that \eqref{series1} converges. 
To do so, we split the sum over $k$ in partial sums running
from $2k_{n-1}$ to $2k_{n}$ and then sum over $n$, i.e.
\[
\sum_{k=1}^{\infty} 2^{3k} g\left(\frac{\Psi(2^k)}{2^k}\right)=
\sum_{n=1}^{\infty} \sum_{k=2k_{n-1}+1}^{2k_{n}} 2^{3k} g\left(\frac{\Psi(2^{k})}{2^{k}}\right).
\]
Now by the definition of $\Psi$, its value at $2^{k}$
is locally constant
$Q_{n}^{2-2\beta}$ for $2k_{n-1}+1\leq k\leq 2k_{n}$. Thus by the
monotonicity of $g$ and since $$\sum_{k=2k_{n-1}+1}^{2k_{n}} 2^{3k}\leq \sum_{k=1}^{2k_{n}} 2^{3k}  \leq 2\cdot 2^{6k_{n}}=2Q_{n}^{6}$$
using $k_0=0$ we conclude that
\begin{align*}
\sum_{k=1}^{\infty} 2^{3k} g(\Psi(2^k)/2^k)&= \sum_{n=1}^{\infty} 
\sum_{k=2k_{n-1}+1}^{2k_{n}} 2^{3k} g\left(\frac{Q_{n}^{2-2\beta}}{2^{k}}\right) \\
&\leq \sum_{n=1}^{\infty} 
\sum_{k=2k_{n-1}+1}^{2k_{n}} 2^{3k} g(Q_{n}^{2-2\beta})
\\ &\leq
\sum_{n=1}^{\infty} 
 2Q_{n}^{6} g(Q_{n}^{2-2\beta}).
\end{align*}
By \eqref{zweite} with $\gamma=2\beta-2>\beta$ we infer
\begin{align*}g(Q_{n}^{-\gamma})&=g(Q_{n}^{2-2\beta})\\ &\leq \max\{ Q_{n+1}^{-2}, Q_{n}^{-\beta}\}\\ &\leq Q_{n}^{-\min\{2\alpha,\beta\}}.\end{align*} Combining our
observations yields
\[
\sum_{k=1}^{\infty} 2^{3k} g(\Psi(2^k)/2^k) \leq
\sum_{n=1}^\infty2Q_{n}^{6-\min\{ 2\alpha,\beta \}}.
\]
Since $\min\{ 2\alpha,\beta \}>6$ by assumption and $(Q_{n})_{n\geq 1}$ grows exponentially, the series converges.


\subsection{Proof of Corollary~\ref{korola}} Similar to the proof of Theorem~\ref{thm2}, the proof of Corollary~\ref{korola} also uses the one-dimensional
divergence part of  Theorem~\ref{JSDV}. Notice that their claims
can be equivalently formulated with $x$ restricted to any set with
nonempty interior, we do not carry this out. It will be convenient
to choose the interval $I=[0,1/2)$. Denote $\Phi(q)=\sqrt{\Psi(q^{2})}$ and notice that $\Phi$ decreases since $\Psi$ does. Moreover
\[
\Theta= \left\{ x\in\left[0,\frac{1}{2}\right):\; \vert vx-u\vert\leq \Phi(v) \ \text{for i. m.} \  (u, v)\in \Z^2\right\},
\]
where it is understood that $u=u_x, v=v_x$ depend on $x$.
By Theorem~\ref{JSDV}, which we may apply as $\Phi$ decreases, we have $\HH^{g}(\Theta)=\infty$ as soon as \[\sum_{q=1}^{\infty} q g\left(\frac{\Phi(q)}{q}\right)\asymp\sum_{k=1}^{\infty} 2^{2k} g\left(\frac{\sqrt{\Psi(2^{2k})}}{2^{k}}\right)=\infty.
\] 
Then $x\in \DD(\VV_2, \Phi)=\{x\in[0,1) : (x,x^2) \in \DD_2(\Phi)\} $ if and only if there exist infinitely many integral quadratic polynomials $P$ such that
\[
|P(x)| \leq \Phi(H_{P}).
\]
We claim that
\[
 \DD(\VV_2, \Phi)\supset \Theta.
\]
If this is true then clearly $\HH^g(\DD(\VV_2, \Phi) )\geq \HH^{g}(\Theta)=\infty$.
Fix any $x\in \Theta$.
Define the polynomial $Q(t)=Q_{u,v}(t)=vt-u$ 
for any $u,v$ as in the
definition of $\Theta$. Since $x<1/2$ and $\Phi$ tends to $0$, if the height $H_{Q}$ is sufficiently large
we have $2\vert u\vert< \vert v\vert$. Thus
$H_{Q}=\max\{ \vert u\vert, \vert v\vert\}= \vert v\vert$. 
Moreover, if we let $P(t)= Q(t)^{2}= v^{2}t^{2}-2uvt+u^{2}$, then
$P$ has height $H_{P}=v^{2}$ provided $H_{P}$ (or $H_{Q}$) 
was chosen large enough.
We verify that
\[
\vert P(t)\vert= \vert Q(t)\vert^{2}\leq \Phi(v)^{2}=\Psi(v^{2})=\Psi(H_{P}).
\]
Since this holds for almost all polynomials $Q(t)$ 
derived from the definition of $\Theta$, we have $x\in  \DD(\VV_2, \Phi)$.
Since $x\in \Theta$ was arbitrary, indeed $\DD(\VV_2, \Phi) \supset \Theta$ 
and the proof is complete.

\subsection{Proof of Theorem~\ref{thm4}}

For this proof we use a property of polynomials 
sometimes referred to as Gelfond's lemma, see \cite[Lemma A.3]{Bugeaud}. 
Recall the notation $H_{P}$ for the height of a polynomial $P\in\mathbb{Z}[t]$.

\begin{lemma}[Gelfond]\label{gelfondl} For any finite set of polynomials $P_{1},\ldots,P_{k}$ 
of heights
$H_{P_{i}}$ for $1\leq i\leq k$,
whose product has degree at most $n$, we have
\begin{equation} \label{gelf}
H_{P_{1}}H_{P_{2}}\cdots H_{P_{k}} \lessless_{n} H_{P_{1}P_{2}\cdots P_{k}}
\lessless_{n} H_{P_{1}}H_{P_{2}}\cdots H_{P_{k}},
\end{equation}
where the implied constants depend on $n$ only.

\end{lemma}

The implied factor $c(n)$ in Lemma~\ref{gelfondl} causes 
inconvenient technical difficulties in view of our application to very general
classes of approximating functions.
Indeed, if we could assume $H_{PQ}\geq H_{P}H_{Q}$, the proof 
of Theorem~\ref{thm4} below would considerably
simplify and we could drop the conditions \eqref{gleichung} and \eqref{gl22} respectively
in its two claims.

\begin{proof}[Proof of Theorem~\ref{thm4}]
Fix $x\in \R$ not algebraic of degree $n$ or less.
By assumption the set
\[
Z_{\Psi}^{\prime}(x):=\left\{ P(t)= a_{n}t^{n}+\cdots +a_{0}: 
\; \vert P(x)\vert\leq \Psi(H_{P})\right\}
\]
is infinite.  
We first show that infinitely many
$P\in Z_{\Psi}^{\prime}(x)$ can be chosen with coprime coefficients, i.e. $\eta(P):=(a_{0},\ldots,a_{n})=1$. 
Indeed, if $P\in Z_{\Psi}^{\prime}(x)$ then also 
$\tilde{P}=P/\eta(P)\in Z_{\Psi}^{\prime}(x)$ since 
\[
\vert \tilde{P}(x)\vert=\frac{\vert P(x)\vert}{\eta(P)}
\leq \frac{\Psi(H_{P})}{\eta(P)}= 
\frac{\Psi(\eta(P)H_{\tilde{P}})}{\eta(P)}
\leq
\frac{\Psi(H_{\tilde{P}})}{\eta(P)}\leq \Psi(H_{\tilde{P}}),
\]
where we used that $\Psi$ decays. On the
other hand, since $\tilde{P}(x)\neq 0$ by our assumption on $x$,
only finitely many scalar multiples $P=\eta(P)\tilde{P}$ of any $\tilde{P}$ 
can be in $Z_{\Psi}^{\prime}(x)$
since $$\vert P(x)\vert=  \eta(P)\vert\tilde{P}(x)\vert>\Psi(H_{\tilde{P}})>\Psi(H_{P})$$
as soon as $\eta(P)> H_{\tilde{P}}/\vert\tilde{P}(x)\vert$. 
Hence, if
the total number of $\tilde{P}$ were finite, so would be $Z_{\Psi}^{\prime}(x)$, contradicting
our hypothesis. The claim is shown.

Denote the corresponding infinite subset of $Z_{\Psi}^{\prime}(x)$ with
coprime coefficients by $Z_{\Psi}(x)$.
For arbitrary $P\in Z_{\Psi}(x)$ of large height,
denote its factorization
over $\mathbb{Z}[t]$ by
\begin{equation} \label{prod}
P(t)= Q_{1}(t)Q_{2}(t)\cdots Q_{k}(t), \qquad Q_{i}\in\mathbb{Z}[t],\; k=k(P).
\end{equation}
Here the  polynomials $Q_{i}$ are irreducible and not necessarily distinct, and none is constant.
By the pigeonhole principle, infinitely many such $P\in Z_{\Psi}(x)$ must have
the same $k=k(P)$, so we may fix $k$ and consider only those $P$ with this factortization
pattern. Let us denote the corresponding subset of $Z_{\Psi}(x)$ by $Z_{\Psi,k}(x)$.
We may assume $k\geq 2$, otherwise claims \eqref{eqq},
\eqref{gl3} are trivial 
with $Q=P$ and $C=1$ and $\epsilon=0$. 

Now for technical reasons we define the following quantity: 
Let $l\in \{0,1,\ldots,k\}$
be the maximum integer with the property that 
for any polynomials $P$ in $Z_{\Psi,k}(x)$
each factoring as in \eqref{prod} into irreducible factors,
there are at least $l$ among its factors $Q_i$ 
which are of absolutely bounded height.
First notice that by Gelfond's estimate we have $l<k$ since
if all heights of $Q_{1},\ldots,Q_{k}$ were absolutely 
bounded then so would be any possible product,
contradicting the assumption that $Z_{\Psi,k}(x)$ is of infinite cardinality. 

The defintion of $l$ guarantees that for some infinite
subset $Z_{\Psi,k,l}(x)$ of polynomials $P$
in $Z_{\Psi,k}(x)$ we have (after relabeling 
of indices if necessary) that 
the induced $Q_i$ in \eqref{prod} satisfy that the heights of
$Q_{1},\ldots, Q_{l}$ are absolutely bounded, whereas 
$H_{Q_i}\to \infty$ for $i\in \{ l+1,l+2,\ldots,k\}$.

Assume we are given 
$P\in Z_{\Psi,k,l}(x)$ with factorization into $Q_{i}$ as in \eqref{prod}
and let $H=H_{P}$.
Then since $Q_{1},\ldots,Q_{l}$ are of absolutely bounded height
there are only finitely many choices. Thus, as $x$ is not algebraic 
of degree $n$ or less, 
for some positive constants $H_{x}, D_{x}$ depending only on $n$ and $x$
but independent from our choice of $P\in Z_{\Psi,k,l}(x)$ 
we have 
\begin{equation} \label{hase}
1\leq H_{Q_{1}}H_{Q_{2}}\cdots H_{Q_{l}}\leq H_{x}, \qquad
\vert Q_{1}(x)Q_{2}(x)\cdots Q_{l}(x)\vert \geq D_{x}.
\end{equation}

For simplicity write $R_{i}(x)=Q_{l+i}(x)$ for $1\leq i\leq k-l$ and let
\[
R(x):=R_{1}(x)\cdots R_{k-l}(x)=Q_{l+1}(x)Q_{l+2}(x)\cdots Q_{k}(x).
\]
Recall
that this product is nonempty for $P\in Z_{\Psi,k,l}(x)$
of sufficiently large height $H=H_{P}$ since $k>l$. Then with
$c=c(n)$ the implied constant in \eqref{gelf} we have
$$\frac{H}{cH_{x}}\leq H_{R}\leq cH$$ and 
\begin{equation} \label{eq:prod2}
\vert R(x)\vert=\frac{\vert P(x)\vert}{\vert Q_{1}(x)\vert \cdots \vert Q_{l}(x)\vert}\leq \frac{\Psi(H)}{D_{x}}.
\end{equation}
Write $H_{i}=H_{R_{i}}$ for simplicity. By Gelfond's Lemma equation \eqref{gelf},
for the heights $H_{i}$ we
have
\begin{equation} \label{eq:kochk}
\prod_{i=1}^{k-l} H_{i}\leq cH_{R}\leq c^{2}\cdot H.
\end{equation}
In view of \eqref{eq:prod2}, 
we may write
\begin{equation} \label{eq1}
\vert R_{i}(x)\vert = |R(x)|^{g_{i}} \leq 
\frac{\Psi(H)^{g_{i}}}{D_{x}^{g_{i}}},
\qquad\qquad  1\leq i\leq k-l,
\end{equation}
and similarly by \eqref{eq:kochk}
\begin{equation} \label{eq2}
H_{i}\leq (c^{2}H)^{h_{i}}= c^{2h_{i}}\cdot H^{h_{i}}, \qquad\qquad  1\leq i\leq k-l,
\end{equation} 
where $g_{i}\geq 0, h_{i}\geq 0$ are real numbers that satisfy
\begin{equation} \label{freiheit}
\sum_{i=1}^{k-l} g_{i}=\sum_{i=1}^{k-l} h_{i}=1.
\end{equation}
%

In view of \eqref{freiheit}, there exists some index $1\leq j\leq k-l$ with
$g_{j}/h_{j}\geq 1$.
Without loss of generality assume
\begin{equation} \label{wloss}
\frac{g_{1}}{h_{1}}\geq 1.
\end{equation}
From \eqref{eq1}, \eqref{eq2} and since $\Psi$ decreases we see
\begin{equation} \label{abschae}
\vert R_{i}(x)\vert \leq \frac{\Psi(H_{i}^{1/h_{i}}/c^{2})^{g_{i}}}{D_{x}^{g_{i}}}, \qquad  1\leq i\leq k-l.
\end{equation}
%
First assume the condition \eqref{gleichung} on $\Psi$. 
Applied to $c_{1}=1/c^{2}$ we see that
\[
\vert R_{i}(x)\vert \leq \frac{c_{2}^{g_{i}}}{D_{x}^{g_{i}}}\cdot \Psi(H_{i}^{1/h_{i}})^{g_{i}}\leq c_{3}\cdot\Psi(H_{i}^{1/h_{i}})^{g_{i}}, \qquad  1\leq i\leq k-l,
\]
with $c_{3}=c_{3}(x)=\max\{ 1,c_{2}/D_{x}\}$ independent from $g_{i}$.
By assumption $\Psi$ satisfies \eqref{tauu}, 
which we observe as equivalent to
\begin{equation} \label{h}
\Psi(q^{\alpha})\leq \Psi(q)^{\alpha}, \qquad \text{ for any}\; \alpha\geq 1 \; \text{and large} \; q.
\end{equation}
With $\alpha=1/h_{i}\geq 1$
and by \eqref{wloss}, and as we can assume $\Psi(H_{1})<1$, we derive
\begin{align*}
\vert R_{1}(x)\vert &\leq c_{3}(x)\Psi(H_{1}^{1/h_{1}})^{g_{1}} \\ &\leq
c_{3}(x)\Psi(H_{1})^{g_{1}/h_{1}}\\ &\leq c_{3}(x)\Psi(H_{1}).
\end{align*}
By assumption $R_{1}=Q_{l+1}$
is irreducible. Moreover, we see that there are
infinitely many pairwise distinct $R_{1}$ arising in this way as the definition of $l$ 
implies $H_{1}=H_{R_{1}}\to \infty$. 
The proof of the claim is finished.
\end{proof}


\subsection{Proof of Theorem~\ref{weil}}
	The proof is similar to the first 
	claim of Theorem~\ref{thm4}, observing
	that the multiplicativity of the
	Mahler measure allows us to let $c_{1}=c_{2}=1$ and this leads us to $c_{3}(x)=\max\{1,1/D_{x}\}$. If the property $(\ast)$ holds
	then $D_{x}$ is the empty product so $C=c_{3}(x)=1$.


%
%

%
%
%
 %

\subsection{Proof of Corollary~\ref{obn}}
 We readily check that the functions that 
	satisfy \eqref{tauu} and \eqref{gleichung}
	are closed under multiplication and any function 
	$q\mapsto q^{-\tau}$ satisfies both of these conditions. We check the conditions for any $\varphi=\varphi_{i}$ as given in the statement of the corollary. 
	First we show that the function $q\mapsto \log \varphi(q)/\log q$ decays
	for $q\geq q_{0}$.
	Since $\varphi(q)>1$ for $q\geq q_{0}$,
	this is clear if $\varphi$ decays. If $\varphi$ is increasing
	and differentiable, then
	$(\log \varphi)^{\prime}(q)>0$ as well and differentiating
	$q\mapsto \log \varphi(q)/\log q$ with the chain rule, its decay thus
	gives the equivalent condition
	$$\frac{\log \varphi(q)}{(\log \varphi)^{\prime}(q)}> q\log q.$$ 
	We readily verify this inequality via using our assumptions, 
	which in particular
	imply $\log \varphi(q)\geq c_{0}>0$ for $q\geq q_{0}$ 
	and $(\log \varphi)^{\prime}(q)>0$ for large $q$.
	Thus \eqref{tauu} holds in both cases.  
	If $\varphi$ increases, then $\varphi(q/c_{1})\leq \varphi(q)$ for $c_{1}>1$ 
	so we may put $c_{2}=1$
	in \eqref{gleichung}. On the other hand, if $\varphi$ decreases for $q\geq q_{0}$,
	then by our first assumption $\varphi(q)>1$. It is easy to see
	that we cannot have $\varphi(q/c_{1})\geq 2\varphi(q)$ for 
	some $c_{1}>1$ and infinitely many arbitrarily
	large $q$. Thus for large $q$ we can take $c_{2}=2$ and \eqref{gleichung} holds in both cases as well.


\subsection{Proof of Theorem~\ref{bugresu}}  \label{bgrs}

For the proof of Theorem~\ref{bugresu} we employ a result of
Davenport and Schmidt \cite[Lemma~8]{MR246822} related to Liouville's inequality~\cite[Corollary~A2]{Bugeaud} on the distance
of two algebraic numbers. It states
that for any integer $n\geq 1$ and any real number $x$, there is a constant $c=c(n,x)>0$ 
so that for any two integer polynomials 
$P,Q$ of degree at most $n$ and without common factor,
if we let $H=\max \{ H_{P},H_{Q}\}$ then 
\begin{equation} \label{eq:buschlei}
\max \{ |P(x)|, |Q(x)| \} \geq c H^{-(2n-1)}. 
\end{equation}
See also Lemma 3.1 in the paper of Bugeaud and Schleischitz \cite{MR3557124} for a refinement.

\medskip

	Fix some $w>n(2n-1)$ and consider the number $x=\xi_{w}>0$
	as constructed in the proof of Corollary~1 in
	\cite{MR2791654}. 
	The key feature of these numbers is that the partial quotients of $\xi_{w}=[a_{0};a_{1},\ldots]$ satisfy the growth condition $a_{i+1}\asymp a_{i}^{w}$. 
	We omit rephrasing the exact construction and only 
	discuss the properties derived in~\cite{MR2791654}. 
	Firstly,
	the  denominators of the 
	convergents $(p_{l}/q_{l})_{l\geq 0}$ to $\xi_{w}$
	have the same growth rate as well, 
	more precisely there exist $c_{2}>c_{1}>0$ such that
	\begin{equation}  \label{eq:ivo}
	c_{1}q_{l-1}^{w} < q_{l} < c_{2}q_{l-1}^{w}, \qquad\qquad l\geq 1.
	\end{equation}
	By fundamental results on continued fractions, we infer that when
	we let  $R_{l}(t):= q_{l}t-p_{l}$ then for some $c_{4}>c_{3}>0$ we have
	\begin{equation}  \label{eq:cofr}
	c_{3}q_{l}^{-w} < \vert R_{l}(\xi_{w})\vert < c_{4}q_{l}^{-w}, \qquad\qquad l\geq 1.
	\end{equation}
	For convenience we may assume
	$H_{R_{l}^n}=H_{R_{l}}^n=q_{l}^n$, in particular $q_{l}>p_{l}>0$.
	Indeed, by $$H_{R_{l}^{n}}= \max\{q_{l}^{n},np_{l}q_{l}^{n-1},\ldots,p_{l}^{n}\}$$
	we see that
	this can be arranged by altering the first few partial quotients of $\xi_{w}$ so that it is close enough to $0$, if needed (alternatively
	we may introduce some additional factor in the proof below since
	$H_{R_{l}^n}\lessless_{n,x} q_{l}^n$ 
	is obvious). 
	
	Define $\Psi=\Psi_{\xi_{w}}$ as the locally constant function that takes the value
	$\Psi(q)=c_{4}^{n}q_{l}^{-wn}$ when $q\in (q_{l-1}^n, q_{l}^n]$,
	for $l\geq 1$.
	Since $q_{l}$ tends monotonically to infinity, $\Psi$ decreases
	to $0$ as desired.
	It has discontinuities at points $q_{l}^{n}$. 
	Notice further that $q_{l}\in (q_{l-1}^n, q_{l}^n]$ in view
	of \eqref{eq:ivo} and $w>n(2n-1)>n$, therefore
	\begin{equation}  \label{eq:ivo2}
	\Psi(q_{l})= c_{4}^{n}q_{l}^{-wn}.
	\end{equation}
	By construction,
	the reducible polynomial $R_{l}^n(t)$ has degree $n$ and 
	height $H_{R_{l}^n}=H_{R_{l}}^n=q_{l}^n$. Furthermore, by \eqref{eq:cofr} it satisfies
	\[
	\vert R_{l}^n(\xi_{w})\vert= \vert R_{l}(\xi_{w})\vert^n < c_{4}^n q_{l}^{-nw}=\Psi(q_{l}^n)= \Psi(H_{R_{l}^n}).
	\]
	This verifies the first claim of the theorem for the 
	reducible polynomials $R_{l}^n$.
	For the second claim, take $Q(t)$ any irreducible polynomial.
	If $Q=R_{l}$ for some $l$, then 
	since $H_{R_{l}}=\max\{ p_{l}, q_{l}\}=q_{l}$
	and by \eqref{eq:cofr}, \eqref{eq:ivo2} we readily check that
	\[
	|Q(\xi_{w})| = | R_{l}(\xi_{w})| > c_{3}q_{l}^{-w} = c_{5}(c_{4}^{n}q_{l}^{-wn})^{1/n} = c_5 \Psi(q_{l})^{1/n} = c_5 \Psi(H_{Q})^{1/n}
	\]
	for the constant $c_{5}=c_{3}/c_{4}>0$. 
	We may put $\Delta(n)=c_5$. Now assume otherwise $Q$ is not among the $R_{l}$.
	Then by its irreducibility, it is coprime to all these polynomials. 
	Let $l$ be the index 
	with $H_{Q}\in (q_{l-1}^n,q_{l}^n]$.
	On the one hand, by definition of $\Psi$ we have
	\begin{equation}  \label{eq:jjj}
	\Psi(H_{Q}) = c_{4}^{n}q_{l}^{-wn}.
	\end{equation}
	On the other hand, since $R_{l}$ and $Q$ are coprime
	and $H_{R_{l}^n}\geq H_{Q}$, from \eqref{eq:buschlei}  for some
	$c_{8}>0$ we get
	\[
	\max\{ |R_{l}^n(\xi_{w})|, |Q(\xi_{w})|  \} \geq c_{8} H_{R_{l}^n}^{-(2n-1)}
	= c_{8}q_{l}^{-n(2n-1)}.  
	\]
	Now since $w>n(2n-1)>2n-1$, for large $l$ from \eqref{eq:cofr} we infer 
	\[
	|R_{l}^n(\xi_{w})|< c_4^n q_l^{-wn} < c_{8}q_{l}^{-n(2n-1)}, \qquad\qquad l\geq l_{0}.
	\]
	Therefore for large $l$ we must have
	\[
	|Q(\xi_{w})| \geq c_{8}q_{l}^{-n(2n-1)}= c_{9} \Psi(H_{Q})^{ \frac{2n-1}{w} },
	\]
	where we used \eqref{eq:jjj} and let $c_{9}=c_{8}c_{4}^{-n(2n-1)/w}>0$. Since we may assume $\Psi(H_{Q})<1$,
	our choice $w>n(2n-1)$ guarantees that the right hand 	side exceeds $c_{9}\Psi(H_{Q})^{1/n}$ again, so we may take $\Delta(n)=c_{9}$. Clearly we can just decrease $\Delta(n)$, if needed,
	to deal with small $l$. This proves the second claim.
	%
	%
	%
	%
	%
	\begin{remark} Our choice of locally constant, discontinuous $\Psi$
	is just for simplicity. It is easy to alter $\Psi$ to derive an arbitrarily
	smooth strictly decaying function that still satisfies the claim of the theorem.
	\end{remark}

\medskip

The proof can be extended to show the claims for
any $x=[a_{0};a_{1},\ldots]$ with convergent sequence $[a_{0};a_{1},\ldots,a_{i}]=p_{i}/q_{i}$
with the property $\liminf_{i\to\infty} \log a_{i+1}/\log q_{i}>n(2n-1)-1$,
or equivalently $\liminf_{i\to\infty} \log q_{i+1}/\log q_{i}>n(2n-1)$.
The set of such numbers has Hausdorff dimension $(2n^2-n+1)^{-1}$
due to Tan--Zhou \cite{MR3956788}.
In the example constructed in the proof above, the function
$\Psi$ depends on $x=\xi_{w}$.
With some more effort, the proof can be refined to provide 
uncountably many $x$
with the properties of Theorem~\ref{bugresu}
for a fixed approximating function $\Psi$. Possibly the same can be obtained
for a set of Hausdorff dimension $(2n^2-n+1)^{-1}$ as above. 

For the proof of Theorem~\ref{abc}, we combine Theorem~\ref{thm4} 
with the method from the proof of~\cite[Corollary~1.6]{Pezzoni}.

\subsection{Proof of Theorem~\ref{abc}}
	Let
	$P_{n,\lambda}^{\ast}\subseteq P_{n,\lambda}$ denote 
	the set of {\em irreducible}
	integer polynomials
	of degree at most $n$ satisfying \eqref{eq:disco} and the set
	$A_{n,\lambda}(\Psi)$ is accordingly altered
	to $A_{n,\lambda}^{\ast}(\Psi)=A_{n,\lambda}^{\ast}(\Psi)$ by restricting to polynomials in $P_{n,\lambda}^{\ast}$. Recall that
	Theorem~\ref{PEZZ} claims that
	if $r\mapsto r^{-1}g(r)$ decreases and $r^{-1}g(r)\to \infty$ as $r\to 0$, then \eqref{eq:PEZ1} holds, which we recall for convenience here:
	\begin{equation} \label{irr}
	\HH^{g}(A_{3,\lambda}^{\ast}(\Psi))=0, \qquad \text{if} \quad
	\sum_{q=1}^{\infty} g\left(\frac{\Psi(q)}{q}\right)q^{3-2\lambda/3}<\infty.
	\end{equation}
	This is the claim of our 
	Theorem~\ref{abc}
	when restricted to irreducible polynomials.
	We must make the transition to considering all polynomials
	on the left hand side.
	
Note that the discriminant condition within the definition
of $A_{n,\lambda}(\Psi)$ is
	automtacially satisfied for $\lambda=0$. Hence,
	similarly to the proof of~\cite[Corollary~1.6]{Pezzoni},
	our Theorem~\ref{thm4} implies
	\begin{equation} \label{union}
	A_{3,\lambda}(\Psi)\subseteq  A_{3,\lambda}^{\ast}(\Psi) \bigcup
	\cup_{N=1}^{\infty}  (A_{1,0}^{\ast}(N\Psi)\cup A_{2,0}^{\ast}(N\Psi)),
	\end{equation}
	since $A_{1,0}^{\ast}(N\Psi)\cup A_{2,0}^{\ast}(N\Psi)$ is the
	set of real numbers $x$ for which we can take the constant $C(3,x)=N$
	in Theorem~\ref{thm4}. Hence,
	by sigma-additivity of measures, it suffices to show that 
	for any $N\geq 1$ the sets 
	\[
	A_{1,0}^{\ast}(N\Psi), \qquad A_{2,0}^{\ast}(N\Psi),\qquad A_{3,\lambda}^{\ast}(\Psi)
	\]
	all have $g$-measure $0$. For $A_{3,\lambda}^{\ast}(\Psi)$ this is just the claim of Theorem~\ref{PEZZ}.
	For the other sets, consider $N$ fixed.
	For the sets $A_{1,0}^{\ast}(N\Psi)$, it follows from
	Theorem~\ref{JSDV} applied to $N\Psi$ that
		\begin{equation*} 
	\HH^{g}(A_{1,0}^{\ast}(N\Psi))=0, \qquad \text{if} \quad
	\sum_{q=1}^{\infty} g\left(\frac{N\Psi(q)}{q}\right)q<\infty.
	\end{equation*}
	By our assumption that $r\mapsto r^{-1}g(r)$ decays we have $g(N\Psi(q)/q)\leq Ng(\Psi(q)/q)$.
	Now because of $3-2\lambda/3\geq 3-2(9/20)/3=27/10>1$,
	our assumption of convergence of the right side in \eqref{irr} 
	indeed implies
	\[
	\sum_{q=1}^{\infty} g\left(\frac{N\Psi(q)}{q}\right)q\leq N\sum_{q=1}^{\infty} g\left(\frac{\Psi(q)}{q}\right)q\leq
	N\sum_{q=1}^{\infty} g\left(\frac{\Psi(q)}{q}\right)q^{3-2\lambda/3}<\infty.
	\]
	Finally for the set with respect to quadratic polynomials,
	the proof of~\cite[Case II]{Hussain} for $N\Psi$ shows
		\begin{equation*} 
	\HH^{g}(A_{2,0}^{\ast}(N\Psi))=0, \qquad \text{if} \quad
	\sum_{q=1}^{\infty} g\left(\frac{N\Psi(q)}{q}\right)q^2<\infty.
	\end{equation*}
	Again similarly as above
	\[
	\sum_{q=1}^{\infty} g\left(\frac{N\Psi(q)}{q}\right)q^2\leq N\sum_{q=1}^{\infty} g\left(\frac{\Psi(q)}{q}\right)q^2.
	\]
	To conclude it suffices to notice that
	\[
	\sum_{q=1}^{\infty} g\left(\frac{\Psi(q)}{q}\right)q^2
	\leq
	\sum_{q=1}^{\infty} g\left(\frac{\Psi(q)}{q}\right)q^{3-2\lambda/3}<\infty
	\]
	since $3-2\lambda/3\geq 27/10>2$ and by our hypothesis
	in \eqref{irr}.

\subsection{Proof of Theorem \ref {Removingmono} }\label{section3} 
%
 Call $\qq=(q_{1},q_{2})\in\mathbb{Z}^2\setminus\{{\0}\}$ a good pair if $\qq= (-2ab,b^2)$ for some $a,b$ coprime with $|a|<b$.
 Let $\Psi$  be a single-variable approximating function, and derive 
 the multi-variable function $\psi$ by letting
  $\psi(\qq) = b^2\Psi^2(b)$ for good pairs $\qq$, and $\psi(\qq) = 0$ otherwise. For good pairs,
  we calculate $$|\qq\cdot(x,x^2)+a^2|=|q_{1}x+q_{2}x^2+a^2|=q_{2}(x-a/b)^2.$$
  Therefore a point $(x,x^2)$ with $x\notin \overline{\mathbb{Q}}$ and $|x|<1$ is $\psi$-approximable if
  there exist infinitely many $a,b$ coprime such that $|a|<b$ and 
 $q_{2} (x - a/b)^2 < b^2\Psi^2(b).$ 
 Rearranging and plugging in $q_{2}=b^2$ gives 
  $|x - a/b| < \Psi(b).$ We apply the divergence case for $n=1$ of 
  Theorem \ref{JSDV}, 
  and notice that obviously coprimality of $a,b$ is not a restriction. 
  It yields that for $s\in(0,1)$, any decreasing $\Psi$ and the dimension function $g(r)=r^s$, the latter set has full $s$-dimensional Hausdorff measure as soon as the series

\begin{equation}\label{eqn1}
\sum_{b=1}^{\infty} b \Psi^s(b) 
 \end{equation}
 diverges. On the other hand, in view of $\|\qq\|=\max\{|q_{1}|,|q_{2}|\}=b^2$ 
 for good pairs,
 the corresponding series for the generalized Baker-Schmidt problem for the multivariable approximating function $\psi$ becomes

\begin{align}\label{eqn2} \sum_{\qq\in \mathbb{Z}^2\setminus\{{{\0}}\} } \|\qq\|^{1-s} \psi(\qq)^s &=
\sum_b {  \sum_{|a|<b, (a,b)=1} } b^{2(1-s)} b^{2s} \Psi^{2s}(b)\\
&\le \sum_b {  \sum_{|a|<b} b^{2(1-s)} b^{2s} \Psi^{2s}(b) \lessless \sum_b b^3 \Psi^{2s}(b)}. \nonumber  
\end{align}

To complete the proof we choose an (ultimately) decreasing approximating function $\Psi$ such that \eqref{eqn2} converges but \eqref{eqn1} diverges, for example
$$  \Psi(q) = q^{-2/s} \log^{-1/s}(q),\;\; (q\geq 2).$$

\medskip

\noindent {\bf Acknowledgements.}  The first-named author was supported by the Australian Research Council Discovery
Project (200100994).  The third-named author was supported by the Royal Society Fellowship. Part of this work was carried out when the second- and third-named authors visited La Trobe University. We are thankful  to the Australian Mathematical Sciences Institute (AMSI) for the travel support. Finally, we thank the anonymous referee for useful comments which has improved the clarity and presentation of the article.

%

\providecommand{\bysame}{\leavevmode\hbox to3em{\hrulefill}\thinspace}
\providecommand{\MR}{\relax\ifhmode\unskip\space\fi MR }
\providecommand{\MRhref}[2]{%
  \href{http://www.ams.org/mathscinet-getitem?mr=#1}{#2}
}
\providecommand{\href}[2]{#2}

\end{document}